\documentclass[10pt]{extarticle}

\usepackage{authblk}
\usepackage{amsmath}
\usepackage{amsfonts}
\usepackage{amsthm}
\usepackage{amssymb}
\usepackage{mathtools}
\usepackage{bm}
\usepackage{booktabs}
\usepackage{enumerate}
\usepackage{enumitem}
\usepackage{color}
\usepackage{comment}
\usepackage[top=1 in,bottom=1 in, left=1 in, right=1 in]{geometry}
\usepackage[backend=biber, style=bwl-FU, sortcites=false, maxcitenames=2, mincitenames=1, maxbibnames=4, uniquelist=false]{biblatex}

\newcommand{\argmin}{\mathop{\rm arg~min}\limits}
\DeclareMathOperator{\tr}{tr}
\newcommand{\R}{\mathbb{R}}
\newcommand{\N}{\mathbb{N}}
\newcommand{\mX}{\bm{X}}
\newcommand{\mY}{\bm{Y}}
\newcommand{\mTheta}{\bm{\Theta}}
\newcommand{\mSigma}{\bm{\Sigma}}
\newcommand{\mI}{\bm{I}}
\newcommand{\mP}{\bm{P}}

\newcommand{\mO}{\bm{O}}

\newcommand{\mR}{\bm{R}}

\newcommand{\mV}{\bm{V}}

\newcommand{\mW}{\bm{W}}
\newcommand{\mM}{\bm{M}}

\newcommand{\mB}{\bm{B}}
\newcommand{\mZ}{\bm{Z}}
\newcommand{\mU}{\bm{U}}
\newcommand{\mXi}{\bm{\Xi}}
\newcommand{\mE}{\bm{\mathcal{E}}}
\newcommand{\mGamma}{\bm{\Gamma}}
\newcommand{\mDelta}{\bm{\Delta}}

\newcommand{\cJ}{\mathcal{J}}

\newcommand{\limlh}{\lim_{\begin{subarray}{c}n\to\infty \\ c_{n,p} \to c_0\end{subarray}}}
\newcommand{\liminflh}{\liminf_{\begin{subarray}{c}n\to\infty \\ c_{n,p} \to c_0\end{subarray}}}
\newcommand{\rank}{\mathop{\rm rank}}
\newcommand{\etr}{\mathop{\rm etr}}
\newcommand{\vecop}{\mathop{\rm vec}}
\newcommand{\pr}{\mathbb{P}}

\newtheorem{definition}{Definition}[section]

\newtheorem{theorem}{Theorem}[section]

\newtheorem{lemma}{Lemma}[section]

\newtheorem{assumption}{Assumption}[section]

\makeatletter
\@addtoreset{equation}{section}

\makeatother

\bibliography{ref.bib}

\title{Consistent Bayesian Information Criterion Based on a Mixture Prior for Possibly High-Dimensional Multivariate Linear Regression Models}
\author{Haruki Kono \thanks{Email: hkono@mit.edu}}
\affil{Department of Economics, MIT}
\author{Tatsuya Kubokawa \thanks{Email: tatsuya@e.u-tokyo.ac.jp}}
\affil{Department of Economics, University of Tokyo}
\date{\today}

\begin{document}

\maketitle

\begin{abstract}
In the problem of selecting variables in a multivariate linear regression model, we derive new Bayesian information criteria based on a prior mixing a smooth distribution and a delta distribution.
Each of them can be interpreted as a fusion of the Akaike information criterion (AIC) and the Bayesian information criterion (BIC).
Inheriting their asymptotic properties, our information criteria are consistent in variable selection in both the large-sample and the high-dimensional asymptotic frameworks.
In numerical simulations, variable selection methods based on our information criteria choose the true set of variables with high probability in most cases.


\end{abstract}

\par\null
\section{Introduction}

Model selection, or variable selection in particular, is one of the most important problems in statistics, and various selection methods have been suggested and studied both theoretically and practically.
See \cite{konishi2008information} for general overview of this field.
In particular, the Akaike information criterion (AIC) and the Bayesian information criterion (BIC), suggested by \cite{akaike1973information} and \cite{schwarz1978estimating}, respectively, have been used in various applications to select appropriate explanatory variables in univariate linear regression models.
It is well known that the BIC is consistent in the sense of choosing the true variables in the conventional large-sample (LS) asymptotic framework (\cite{nishii1984asymptotic}), whereas the AIC is not.
In contrast, \cite{yanagihara2015consistency} established the consistency of the AIC and the inconsistency of the BIC in high-dimensional (HD) cases in variable selection of multivariate linear regression models.
These results have cast doubt on careless use of information criteria especially in the current world where high-dimensional data are ubiquitous due to advances in technology.
Hence, researchers have paid attention to variable selection methods that are consistent both in the LS and HD frameworks.
One simple way to construct such a variable selection method is to adjust the penalty term of the AIC.
Indeed, \cite{yanagihara2019evaluation} proposed a consistent variable selection criterion based on this idea.
However, this approach does not fit the concept of the AIC; namely, the adjusted penalty term is not an exact or approximately unbiased estimator of the bias which arises when the true density is approximated by a predictive density.
This motivates us to build a new information criterion that not only meets the fundamental concept of the AIC but also has the consistency properties in both LS and HD cases.

%\subsection{Multivariate Linear Regression}
To specify the problem we address, suppose that one has $n$ independent observations of $p$ response variables and $k$ explanatory variables.
Let $\mY$ be an $n \times p$ observation matrix of the response variables, and let $\mX$ be an $n \times k$ observation matrix of the explanatory variables with $\rank \mX = k \ (< n)$.
Let $\omega= \{1, 2, \cdots, k\}$.
For a subset $j$ of $\omega,$ $\mX_j$ denotes the $n \times k_j$ matrix consisting of the columns of $\mX$ indexed by the elements of $j,$ and $k_j$ denotes the number of the elements of $j$.
It is noted that $\mX = \mX_\omega$ and $k = k_\omega$.
For $j \subset \omega,$ we consider the following multivariate linear regression model:
\begin{align} \label{eq:mvr}
  \mY \sim N_{n \times p} (\mX_j \mTheta_j, \mSigma_j \otimes \mI_n),
\end{align}
where $\mI_n$ is the $n \times n$ identity matrix, $\mTheta_j$ is a $k_j \times p$ unknown matrix of regression coefficients, and $\mSigma_j$ is a $p \times p$ unknown covariance matrix.
We assume that the observed data are generated from the following true model:
\begin{align*}
  \mY \sim N_{n \times p} (\mX_{j_\ast} \mTheta_{\ast}, \mSigma_{\ast} \otimes \mI_n)
\end{align*}
for some $j_\ast \subset \omega$.
For the simplicity of notation, $\mX_{j_\ast}$ and $k_{j_\ast}$ are abbreviated as $\mX_\ast$ and $k_\ast,$ respectively.
The variable selection problem is the estimation of $j_\ast$ out of the set $\cJ \subset 2^\omega$ of candidate models.
In particular, we focus on model selection methods based on information criteria; that is, we estimate the true model $j_\ast$ with
\begin{align*}
  \hat j_\ast
  =
  \argmin_{j \in \cJ} \text{IC}(j),
\end{align*}
for an information criterion IC.

%\subsection{Conventional Framework of Information Criteria}
Before we start, we review the basic concept of information criteria in the context of the multivariate linear regression model (\ref{eq:mvr}).
We write the probability density function of a matrix normal distribution $N_{n \times p} (\mM, \mSigma \otimes \mI_n)$ as $f(\cdot \mid \mM, \mSigma)$ for an $n \times p$ matrix $\mM$ and a $p \times p$ matrix $\mSigma.$
We estimate the true density $f_\ast = f(\cdot \mid \mX_\ast \mTheta_\ast, \mSigma_\ast)$ with a predictive density $\hat f_j,$ which depends on $\mY.$
We evaluate the performance of $\hat f_j$ with the risk with respect to the Kullback-Leibler (KL) divergence:
\begin{align*}
  %E^{\mY} [KL(f_\ast || \hat f_j)]
  %=
  E^{\mY} \left[\int
    \left\{
      \log\frac{f_\ast(\tilde\mY)}{\hat f_j(\tilde\mY)}
    \right\}
    f_\ast(\tilde\mY) d\tilde\mY
  \right]
  %\\
  =
  \int \{\log f_\ast (\tilde\mY)\} f_\ast(\tilde\mY) d\tilde\mY + E^{\mY, \tilde\mY}[-\log \hat f_j(\tilde\mY)]
  ,
\end{align*}
where the expectation is in the true distribution.
Since the first term is independent of the predictive density of the candidate model, we can focus on the second term to see the performance of $\hat f_j$.
Thus, for the sake of model selection problems, it is reasonable to find a subset $\hat j_\ast \in \cJ$ of $\omega$ that minimizes $E^{\mY, \tilde\mY}[-2\log \hat f_j(\tilde\mY)]$.
However, this model selection method is not feasible because the objective function depends on unknown parameters: $\mTheta_\ast, \mSigma_\ast,$ and $j_\ast.$
Hence, we estimate $E^{\mY, \tilde\mY}[-2\log \hat f_j(\tilde\mY)]$ with an information criterion of the following form:
\begin{align*}
 {\rm IC}(j)
  =
  - 2 \log \hat f_j (\mY) + \hat b_j,
\end{align*}
where $\hat b_j$ is an exact or asymptotically unbiased estimator of $E^{\mY, \tilde\mY} [-2 \log \hat f_j(\tilde\mY) + 2 \log \hat f_j(\mY)].$
To estimate the true model, we select the model that minimizes ${\rm IC}(j).$

In this framework, various types of information criteria are derived from different predictive densities.
For example, from the maximum likelihood estimator (MLE) plug-in predictive density $\hat f_j = f(\cdot \mid \mX_j \hat\mTheta_j, \hat\mSigma_j),$ the exact AIC, or the corrected AIC, which was proposed in \cite{sugiura1978further} and \cite{bedrick1994model}, is induced:
\begin{align*}
  {\rm AIC}_{\rm C}(j)
  =
  -2 \log f(\mY \mid \mX_j \hat\mTheta_j, \hat\mSigma_j) + \frac{np(2k_j + p + 1)}{n - p - k_j - 1}
  .
\end{align*}
Here the MLE's are represented as
\begin{align*}
  \hat \mTheta_j
  =
  (\mX_j^\top \mX_j)^{-1}\mX_j^\top \mY
  ,
  \ \
  \hat \mSigma_j
  =
  \frac{1}{n} \mY^\top  (\mI_n - \mP_j) \mY,
\end{align*}
where $\mP_j$ is the projection matrix:
$\mP_j = \mX_j(\mX_j^\top \mX_j)^{-1}\mX_j^\top.$
When $n$ is sufficiently large, the exact AIC is approximated by the AIC:
\begin{align*}
  {\rm AIC}(j)
  =
  -2 \log f(\mY \mid \mX_j \hat\mTheta_j, \hat\mSigma_j) + 2 \left(k_j p + \frac{p (p+1)}{2}\right)
  .
\end{align*}
Another example is the BIC, which is represented as follows:
\begin{align*}
  {\rm BIC}(j)
  =
  -2 \log f(\mY \mid \mX_j\hat\mTheta_j, \hat\mSigma_j) + (\log n) \left(k_j p + \frac{p (p+1)}{2}\right)
  .
\end{align*}
The BIC is derived when one assumes a smooth prior $\pi$ for the parameters.
In fact, when the marginal likelihood $\hat f_j = \int\int f(\cdot \mid \mX_j \mTheta_j, \mSigma_j) \pi(\mTheta_j, \mSigma_j) d\mTheta_j\mSigma_j$ is used as a predictive density, we obtain the BIC by applying the Laplace approximation to it.

%\subsection{Consistency and Asymptotic Frameworks}
%Outcomes of model selection highly depend on the choice of information criteria.
Among possible information criteria for variable selection, those that select the true model with high probability are desired.
In this sense, a good information criterion is required to exhibit consistency, which is defined as:
\begin{definition}
  A model selection procedure $\hat j_\ast$ is said to be {\it consistent} if $\pr(\hat j_\ast = j_\ast)$ asymptotically goes to one.
\end{definition}

It is noted that the notion of consistency depends on the asymptotic framework one considers.
In many studies, properties of various information criteria have been investigated in the LS asymptotic framework where the data size $n$ grows to infinity while the other parameters are fixed.
In this framework, it is well known that the BIC is consistent whereas neither the AIC nor the exact AIC is, as stated at the beginning of this section.
Although the AIC and the exact AIC are often criticized due to their inconsistency, the model selection methods based on them are designed to minimize the KL risk rather than to choose the true model with high probability.
Indeed, \cite{shibata1981optimal} and \cite{shao1997asymptotic} showed that the variable selection based on the AIC asymptotically minimizes the prediction error.

In practical applications, especially in economics, one often witnesses situations where response variables are high-dimensional compared to the data size.
One of the common examples is the portfolio choice problem in financial economics.
Specifically, Eugene Fama and Kenneth French have investigated linear models to regress returns of multiple portfolios on other economic variables in their studies (e.g., \cite{fama1993common} and \cite{fama2015five}).
Since economists usually pay attention to a lot of portfolios simultaneously, the number of response variables can be as large as that of observations.
Although they could apply classical statistical methods developed in the LS framework to this kind of data, it is known that when the dimension of response variables is not ignorable compared with the data size, approximation errors from such methods are often huge.
This is why a lot of studies have examined properties in the HD asymptotic framework.
Mathematically, this framework handles the case when the dimension $p$ of response variables, as well as the data size $n,$ goes to infinity under the condition that ${p}/{n} \to c_0 \in [0, 1).$
In this framework, \cite{yanagihara2015consistency} established, as we have seen above, a surprising result: in the HD asymptotic framework, the AIC and the exact AIC are consistent while the BIC is not.

To deal with both of the LS and HD asymptotic frameworks at the same time, we consider the following asymptotic framework:
\begin{align} \label{ass:asy}
  n \to \infty
  ,
  \ \
  c_{n, p} = \frac{p}{n} \to c_0 \in[0, 1)
  .
\end{align}
\noindent
This framework has been investigated in many studies, such as \cite{yanagihara2017high} and \cite{yanagihara2019evaluation}.
We write the limit under (\ref{ass:asy}) as
\begin{align} \label{ass:unif_asy}
  \limlh
  ,
\end{align}
and we use the notation $p(n)$ for $p$ when we need to emphasize that it may vary with $n.$
Following the existing literature, we assume that $p(n)$ is nondecreasing, and therefore, $p(n)$ is either bounded or divergent.
The framework (\ref{ass:asy}) contains both LS and HD as its special cases, and moreover, it allows us to handle data with bounded but inconstant $p.$
In this paper, $o(x), O(x), o_p(x),$ and $O_p(x)$ denote the Landau notations in (\ref{ass:unif_asy}).

In practical applications, both AIC and BIC are used for model selection, but they often choose different models.
If analysts are not sure if the data is LS or HD, they cannot choose the true model consistently.
Therefore, it is a natural question whether there exists an information criterion that is consistent in both frameworks, or (\ref{ass:unif_asy}).
\cite{yanagihara2019evaluation} suggested an information criterion based on the GIC proposed by \cite{nishii1984asymptotic}:
\begin{align} \label{eq:gic}
  {\rm GIC}(j)
  =
  -2 \log f(\mY \mid \mX_j \hat\mTheta_j, \hat\mSigma_j)
  +
  \alpha \left(k_j p + \frac{p (p+1)}{2}\right)
\end{align}
where $\alpha = \beta - n\log \left(1 - p/n\right)/p,$ $p^{1/2} \beta \to \infty,$ and $\beta p/n \to 0,$ and showed that it is consistent in (\ref{ass:unif_asy}) under some conditions.

Although the GIC maintains the consistency property, its construction completely ignores the basic concept of an information criterion because it is derived just by artificially adjusting the penalty term of the AIC.
Due to the lack of interpretability of the choice of $\alpha,$ one cannot apply this information criterion in a straightforward way to models other than the multivariate linear regression model.
In this paper, we derive new consistent information criteria within the conventional framework discussed above without arbitrarily adjusting the penalty term.
Our idea is to minimize the KL risk of a predictive density based on the Bayesian marginal likelihood from a frequentist point of view.
To do this, we consider a Bayesian model, assuming a mixture prior between a smooth distribution and a degenerate one for the regression coefficients.
We can calculate the Bayesian marginal likelihood of this model, but it depends on some unknown parameters.
By estimating them with MLE's, we can obtain an estimator of the Bayesian marginal likelihood, and we use it as a predictive density.
This is known as the empirical Bayesian method, and \cite{kawakubo2018variant} suggested some information criteria based on this procedure, although they did not consider mixture prior distributions.
By setting the mixture weight properly, we can see that the information criterion based on this predictive density performs like the BIC and the AIC (or the exact AIC, more precisely) in the LS and HD asymptotic frameworks, respectively.
Hence, our new criterion inherits asymptotic properties of each criterion, and the variable selection based on it is expected to be consistent in both frameworks, which is shown to be the case in the latter part of this paper.

\textcolor{black}{
This paper is organized as follows.
In Section \ref{sec:new_information_criterion}, we propose a new family of information criteria based on a Bayesian model with a mixture prior of a smooth distribution and a degenerate one, as stated above.
We also provide some examples of this by changing the smooth part of the prior.
In Section \ref{sec:consistency}, we extend the results of \cite{yanagihara2015consistency}, which provided sufficient conditions for the consistency properties of variable selection methods.
Besides, we apply them to show that our new information criteria are consistent.
In Section \ref{sec:discussion}, we discuss issues that appear in practical applications.
Specifically, as the performances of our information criteria may depend on the mixture weight of the prior, we suggest some ways to determine it.
In Section \ref{sec:numerical_study}, we provide the results of numerical experiments and show that our new methods overwhelm existing ones.
Section \ref{sec:conclusion} concludes.
All the technical details are provided in the appendix.
}

\section{New Information Criterion} \label{sec:new_information_criterion}
\subsection{Derivation}
Our goal is to construct a new consistent information criterion based on the KL risk.
To do this, we shall invent an information criterion that performs like the BIC when the data size is large enough and like the AIC when the dimension of the response variable is not ignorable.
Recall that the AIC is derived when the true density is estimated by the MLE plug-in density, and that the BIC originated from the Bayesian marginal likelihood when one assumes a smooth prior on parameters.
Hence, we can expect that a consistent information criterion would be derived by estimating the true density with the marginal likelihood of a Bayesian model with a mixture prior of a smooth distribution and a degenerate one.
Let us consider the following Bayesian model based on a mixture prior for $\mTheta_j:$
\begin{align} \label{eq:ss_model}
  \begin{split}
    \hspace{0.55cm} \mY \mid \mTheta_j, \mSigma_{0, j} &\sim N_{n \times p}(\mX_j \mTheta_j, \mSigma_{0, j} \otimes \mI_n)
    \\
    \mTheta_j \mid \mTheta_{0, j}, \mSigma_{0, j}, w_j &\sim (1 - w_j)\pi(\mTheta_j \mid \mSigma_{0, j}) + w_j \delta_{\mTheta_{0, j}}(\mTheta_j)
  \end{split}
  ,
\end{align}
where $w_j \in (0, 1),$  $\pi$ is a smooth density on $\R^{k_j \times p},$  and $\delta_{\mTheta_{0, j}}$ is a degenerate density on $\R^{k_j \times p}$ that concentrates at $\mTheta_{0, j}$.
Here, $\mTheta_{0, j}$ and $\mSigma_{0, j}$ are unknown.

$w_j$ is a parameter that varies depending on properties of the data.
In particular, $w_j$ controls the high-dimensionality of the data.
We set $w_j$ so that when the data size $n$ is large enough, $w_j$ is close to zero, and it grows as the dimension $p$ of the response variable increases.
Therefore, the smooth part dominates (\ref{eq:ss_model}) in LS cases whereas the spike part does in HD cases.
\textcolor{black}{
As is discussed later, a typical choice of $w_j$ is
\begin{align} \label{eq:w}
  w_j
  =
  \frac{p^{\varepsilon k_j}}{n^{\varepsilon k_j} + p^{\varepsilon k_j}}
\end{align}
for some constant $\varepsilon > 0.$
This converges to zero if $c_0 = 0;$ otherwise, the limit remains positive.
}

The idea of mixing a smooth distribution and a degenerate one is not very new and has been investigated in many studies.
For example, \cite{berger1986robust} considered a contamination in parameters, and proposed a prior distribution similar to ours.
Because they introduced it to discuss the robustness of Bayes estimators, the weight assigned to the degenerate distribution is expected to be small, but here $w_j$ can vary freely in $(0, 1).$
Another example of mixture distribution is a spike and slab prior suggested by \cite{ishwaran2005spike} in the context of Bayesian model selection.
In their procedure, a spike and slab prior is set to each coefficient of explanatory variables, and then one estimates all the weights to execute a model selection by regarding parameters with a nonzero weight as effective.
However, we use the mixture prior in a different way from the ordinary Bayesian context.
Here we reckon $\mTheta_j$ as one parameter and assume the mixture prior with respect to it.

On the regularity of $\pi,$  we assume the following technical conditions:
\begin{assumption} \label{ass:pi}
  Let $\pi_M = \sup_{\mTheta \in \R^{k_j \times p}} \pi(\mTheta \mid \mI_p)$ and $\mSigma$ be a $p \times p$ positive definite matrix.
  \\
  (1) $\pi(\mTheta \mid \mSigma) |\mSigma|^{k_j/2} = \pi(\mTheta \mSigma^{-1/2} \mid \mI_p)$ holds for any $k_j \times p$ matrix $\mTheta;$
  \\
  (2) there exists a sequence $\{\ell_{n, p}\}$ such that
  \begin{align*}
      \limlh \ell_{n, p} = \infty
      ,
      \ \
      \sup \frac{(1 - w_j) \ell_{n, p}^{pk_j/2}}{w_j n^{pk_j/2}} \pi_M
      <
      \infty
      ,
  \end{align*}
  where the supremum is taken over $\{(n, p) \in \N^2 \mid p = p(n)\}.$
\end{assumption}
\noindent
The first condition states that $\pi$ is a multivariate scale family.
Of course, a normal distribution $\pi(\mTheta_j \mid \mSigma) = f(\mTheta_j \mid \mO_{k_j \times p}, \mSigma)$ satisfies it, and moreover, a (scaled) uniform distribution $\pi(\mTheta \mid \mSigma) = |\mSigma|^{-k_j/2}$ also meets it, although this is an improper prior.
The second condition allows us to approximately reduce the information criterion that is suggested later to a simpler form and ensures its consistency.
This condition excludes an exceptionally small $w_j,$ which deteriorates the quality of the approximation and makes the penalty term of the information criterion too harsh.
Although this condition is not necessary as we discuss in Section \ref{sec:discussion}, it is not too restrictive and easy to confirm.
For example, take (\ref{eq:w}) for $\varepsilon \in (0, 1/2)$ as $w_j.$
Then it is obvious that $\ell_{n, p} = \log n$ satisfies the condition if $\sup_{p \in \N} \pi_M < \infty.$
In sum, a broad family of prior distributions including the two distributions above meets Assumption \ref{ass:pi}.

Now, we derive a predictive density based on the model (\ref{eq:ss_model}).
The marginal likelihood is
\begin{align*}
  f_\pi(\mY \mid w_j, \mTheta_{0, j}, \mSigma_{0, j})
  =
  (1 - w_j) f_\pi(\mY \mid \mSigma_{0, j})
  +
  w_j f(\mY \mid \mX_j\mTheta_{0, j}, \mSigma_{0, j})
  ,
\end{align*}
where $f_\pi(\mY \mid \mSigma_{0, j}) = \int f(\mY \mid \mX_j\mTheta_j, \mSigma_{0, j}) \pi(\mTheta_j \mid \mSigma_{0, j}) d\mTheta_j.$
To derive a predictive density based on this, we rely on the empirical Bayesian method.
We substitute the MLE's $\hat \mTheta_j$ and $\hat \mSigma_j$ into $\mTheta_{0, j}$ and $\mSigma_{0, j},$ respectively, and set $f_\pi(\mY \mid w_j, \hat \mTheta_j, \hat \mSigma_j)$ as the predictive density.
Then the information criterion induced from this predictive density has the following form:
\begin{align*}
  - 2 \log f_\pi(\mY \mid w_j, \hat \mTheta_j, \hat \mSigma_j) + \hat b_j
\end{align*}
where $\hat b_j$ is an asymptotically unbiased estimator of
\begin{align*}
  b_j
  =
  E^{\mY, \tilde\mY}[- 2 \log f_\pi(\tilde\mY \mid w_j,\hat\mTheta_j, \hat\mSigma_j) + 2 \log f_\pi(\mY \mid w_j, \hat\mTheta_j, \hat\mSigma_j)]
  .
\end{align*}
The following theorem is useful to evaluate $b_j.$
The proof is given in the appendix.

\begin{theorem} \label{thm:bias}
  Assume Assumption \ref{ass:pi}, and suppose that
  \begin{align} \label{eq:x}
    \liminflh \left|\frac{\mX_j^\top \mX_j}{n}\right| > 0
    ,
  \end{align}
  Then we have
  \begin{align*}
    b_j
    =
    \frac{np(2k_j + p + 1)}{n - p - k_j - 1}
    +
    o(p).
  \end{align*}
\end{theorem}

Let us remark that the assumption (\ref{eq:x}) is not very restrictive.
Indeed, it is weaker than the condition that there exists a positive definite matrix $\mR_j$ such that
\begin{align} \label{eq:x_mat}
  \limlh
  \frac{\mX_j^\top \mX_j}{n}
  =
  \mR_j
  ,
\end{align}
which is a common assumption in studies on the asymptotic theory of the linear regression model (e.g., \cite{fujikoshi1997modified}, \cite{fujikoshi2005bias}, and \cite{yanagihara2015consistency}).
Another common assumption is
\begin{align} \label{eq:x_ev}
  \liminflh \lambda_{\text{min}}\left(\frac{\mX_j^\top \mX_j}{n}\right)
  >
  0
  ,
\end{align}
where $\lambda_{\text{min}}(A)$ denotes the smallest eigenvalue of $A.$
This condition is obviously weaker than (\ref{eq:x_mat}).
However, it is sufficient but not necessary for (\ref{eq:x}).
This is because for a square matrix $\mB,$ $\liminflh |\mB| > 0$ does not necessarily imply $ \liminflh \lambda_{\text{min}}(\mB) > 0.$
A simple counterexample of this is
$
  \mB
  =
  \begin{pmatrix}
    n & 0 \\
    0 & 1/n
  \end{pmatrix}
  .
$

Substituting $np(2k_j + p + 1)(n - p - k_j - 1)^{-1}$ into the bias term $\hat b_j,$ we obtain the following information criterion:
\begin{align*}
  {\rm MPIC}_\pi(j)
  &=
  -2 \log f_\pi(\mY \mid w_j, \hat\mTheta_j, \hat\mSigma_j) + \frac{np(2k_j + p + 1)}{n - p - k_j - 1}
  %\\
  %&=
  %-2 \log \left\{(1 - w_j) f_\pi(\mY \mid \hat\mSigma_j) + w_j f(\mY \mid \mX_j \hat\mTheta_j, \hat\mSigma_j)\right\} + \frac{np(2k_j + p + 1)}{n - p - k_j - 1}
  %\\
  \\
  &=
  n \log |\hat \mSigma_j|
  +
  np(\log 2\pi + 1)
  +
  m_\pi(j)
  ,
\end{align*}
where
\begin{align*}
  m_\pi(j)
  =
  \frac{np(2k_j + p + 1)}{n - p - k_j - 1}
  -
  2 \log \left\{
    \frac{(1 - w_j)f_\pi(\mY \mid \hat \mSigma_j)}{f(\mY \mid \mX_j \hat\mTheta_j, \hat \mSigma_j)}
    +
    w_j
  \right\}
  .
\end{align*}
The MPIC stands for the Mixture Prior Information Criterion.
Although the difference between $b_j$ and $np(2k_j + p + 1)(n - p - k_j - 1)^{-1}$ may not vanish when $p \to \infty,$ the terms of the order $o(p)$ do not matter in variable selection because the terms related to $k_j$ are at least $O(p).$

One can interpret the $\text{MPIC}_\pi$ as a mixture of existing information criteria.
When $w_j$ is close to one,  which is an HD case, $\text{MPIC}_\pi$ is almost the exact AIC.
When $w_j$ is close to zero, which is an LS case, on the other hand, it is reduced to a variant of the BIC, or more precisely, the ABIC suggested by \cite{akaike1980seasonal}, by applying the Laplace approximation.
Given these facts, one can intuitively expect that $\text{MPIC}_\pi$ performs well in both the LS and HD asymptotic frameworks, or under (\ref{ass:unif_asy}), if $w_j$ is properly set, which is indeed true as is shown in Section \ref{sec:consistency}.

\subsection{Examples}
The $\text{MPIC}_\pi$ depends on the choice of $\pi$.
Here we provide two examples of the information criteria by considering different distributions as the smooth part of the mixture prior.
Notice that we can specify the second term of $f_\pi(\mY \mid w_j, \hat\mTheta_j, \hat\mSigma_j),$ which is independent of $\pi,$ as follows:
\begin{align*}
  f(\mY \mid \mX_j \hat \mTheta_j, \hat\mSigma_j)
  =
  \frac{e^{-np/2}}{(2\pi)^{np/2}|\hat\mSigma_j|^{n/2}}
  .
\end{align*}

At first, consider the case of the normal distribution: $\pi(\mTheta_j \mid \mSigma) = f(\mTheta_j \mid \mO_{k_j \times p}, \mSigma),$ which satisfies Assumption \ref{ass:pi} as stated in the previous subsection.
By the ordinary marginalization procedure, we have
\begin{align*}
  f_\pi(\mY \mid \hat \mSigma_j)
  =
  \frac{\etr(-(1/2) \hat \mSigma_j^{-1}\mY^\top (\mI_n+\mX_j \mX_j^\top )^{-1}\mY)}{(2\pi)^{np/2}|\hat \mSigma_j|^{n/2}|\mI_n+\mX_j \mX_j^\top |^{p/2}}
  ,
\end{align*}
where $\etr(\cdot) = \exp(\tr(\cdot)).$
Thus, the information criterion we propose for the normal distribution is
%\footnotesize
\begin{align*}
  {\rm MPIC}_{\text{Normal}}(j)
  %&=
  %-2 \log \left\{
  %\frac{(1 - w_j)\etr(-1/2\hat \mSigma_j^{-1}\mY^\top (\mI_n+\mX_j \mX_j^\top )^{-1}\mY)}{\sqrt{2\pi}^{np}|\hat \mSigma_j|^{n/2}|\mI_n+\mX_j \mX_j^\top |^{p/2}}
  %+
  %\frac{w_je^{-np/2}}{\sqrt{2\pi}^{np}|\hat\mSigma_j|^{n/2}}
  %\right\}
  %+
  %\frac{np(2k_j + p + 1)}{n - p - k_j - 1}
  %\\
  &=
  n \log|\hat\mSigma_j|
  +
  np (\log 2\pi + 1)
  +
  m_{\text{Normal}}(j)
  ,
\end{align*}
where
\begin{align*}
  m_{\text{Normal}}(j)
  =
  \frac{np(2k_j + p + 1)}{n - p - k_j - 1}
  -2 \log \left\{
  \frac
  {(1 - w_j)\etr((np/2) \mI_p - (1/2)\hat \mSigma_j^{-1}\mY^\top (\mI_n+\mX_j \mX_j^\top )^{-1}\mY)}
  {|\mI_n+\mX_j \mX_j^\top |^{p/2}}
  +
  w_j
  \right\}
  .
\end{align*}
%\normalsize

The second example is for the uniform distribution: $\pi(\mTheta_j \mid \mSigma) = |\mSigma|^{-k_j/2},$  which also meets Assumption \ref{ass:pi}.
In this case, the marginal likelihood is
\begin{align*}
  f_\pi(\mY \mid \hat \mSigma_j)
  =
  \frac{e^{-np/2}}{(2\pi)^{(n-k_j)p/2}|\hat\mSigma_j|^{n/2}|\mX_j^\top \mX_j|^{p/2}}
  ,
\end{align*}
and therefore, the information criterion for the uniform prior is
%\footnotesize
\begin{align*}
  {\rm MPIC}_{\text{Uniform}}(j)
  %&=
  %-2 \log \left\{
  %\frac
  %{(1 - w_j)e^{-np/2}}
  %{\sqrt{2\pi}^{(n-k_j)p}|\hat\mSigma_j|^{\frac{n-k_j}{2}}|\mX_j^\top \mX_j|^{p/2}}
  %+
  %\frac
  %{w_je^{-np/2}}
  %{\sqrt{2\pi}^{np}|\hat\mSigma_j|^{n/2}}\right\}
  %+
  %\frac{np(2k_j + p + 1)}{n - p - k_j - 1}
  %\\
  &=
  n \log|\hat\mSigma_j|
  +
  np (\log 2\pi + 1)
  +
  m_{\text{Uniform}}(j)
  ,
\end{align*}
where
\begin{align*}
  m_{\text{Uniform}}(j)
  =
  \frac{np(2k_j + p + 1)}{n - p - k_j - 1}
  -
  2 \log \left\{
    \frac
    {(1 - w_j)(2\pi)^{k_jp/2}}
    {|\mX_j^\top \mX_j|^{p/2}}
    +
    w_j
  \right\}
  .
\end{align*}
%\normalsize

Although the $\text{MPIC}_\pi$ depends on $\pi$ as we have seen in these two examples, it has an approximation that is independent of $\pi$.
We can rewrite the $m_\pi$ as
%\footnotesize
\begin{align*}
  %{\rm MPIC}_\pi(j)
  m_\pi(j)
  =
  %-
  %2 \log f(\mY \mid \mX_j \hat \mTheta_j, \hat \mSigma_j)
  %n \log|\hat\mSigma_j|
  %+
  %np (\log 2\pi + 1)
  %+
  \frac{np(2k_j + p + 1)}{n - p - k_j - 1}
  -
  2 \log w_j
  -
  2 \log \left\{1 + \frac{(1 - w_j) f_\pi(\mY \mid \hat\mSigma_j)}{w_j f(\mY \mid \mX_j \hat \mTheta_j, \hat\mSigma_j)}\right\}
  .
\end{align*}
%\normalsize
On the last term, we can show the following lemma.
The proof is given in the appendix.

\begin{lemma} \label{lem:bf}
  Under Assumption \ref{ass:pi} and (\ref{eq:x}),
  \begin{align*}
    \log \left\{1 + \frac{(1 - w_j) f_\pi(\mY \mid \hat\mSigma_j)}{w_j f(\mY \mid \mX_j \hat\mTheta_j, \hat\mSigma_j)}\right\}
    =
    o_p(1)
  \end{align*}
  holds.
\end{lemma}
\noindent

From Lemma \ref{lem:bf}, $\text{MPIC}_\pi$ can be approximated by
\begin{align*}
  {\rm MPIC}_{\text{Approx}}(j)
  &=
  %-
  %2 \log f(\mY \mid \mX_j \hat \mTheta_j, \hat \mSigma_j)
  %+
  %\frac{np(2k_j + p + 1)}{n - p - k_j - 1}
  %-
  %2 \log w_j
  %\\
  %&=
  n \log|\hat\mSigma_j|
  +
  np (\log 2\pi + 1)
  +
  m_{\text{Approx}}(j)
  ,
\end{align*}
where
\begin{align*}
  m_{\text{Approx}}(j)
  =
  \frac{np(2k_j + p + 1)}{n - p - k_j - 1}
  -
  2 \log w_j
  .
\end{align*}
This information criterion is independent of the choice of $\pi$.
A similar approximation procedure also appears in the derivation of the BIC, where one first computes the marginal likelihood and then shows that the effect of the prior is at most $O_p(1)$ by the Laplace approximation.
Because the $\text{MPIC}_\pi$ is approximated by the same criterion no matter what distribution one uses as the smooth part of the mixture prior, the choice of $\pi$ is not very significant in terms of variable selection.
In other words, the variable selection method based on the MPIC is robust to a prior misspecification.

\section{Consistency} \label{sec:consistency}
As explained in the last section, the new information criterion $\text{MPIC}_\pi$ is regarded as a mixture of the exact AIC and the BIC, or BIC-variants.
Recall that in the LS asymptotic framework, the BIC is consistent while the AIC is not, and vice versa in the HD asymptotic framework.
Thus, if the weight $w_j$ varies appropriately according to the data, we can expect that the $\text{MPIC}_\pi$ is consistent in (\ref{ass:unif_asy}).
In this section, we show that this intuition is actually true under some additional conditions.

We begin by describing some notation.
The set $\cJ$ of candidate models is separated into two parts: $\cJ_+ = \{j \in \cJ \mid j_\ast \subset j\}.$ and $\cJ_- = \{j \in \cJ \mid j_\ast \not\subset j\}.$
We define the noncentrality matrix as
\begin{align*}
  \mSigma_\ast^{-1/2} \mTheta_\ast^\top  \mX_\ast^\top  (\mI_n - \mP_j) \mX_\ast \mTheta_\ast \mSigma_\ast^{-1/2},
\end{align*}
which plays a great role when we investigate properties of the MLE of the covariance matrix.
For $j \in \cJ_+,$ we have $\mP_j\mX_\ast = \mO_{n \times k_\ast},$ and so the noncentrality matrix is always zero.
On the other hand, it may not vanish for $j \in \cJ_-.$
Let $\gamma_j$ be the rank of the noncentrality matrix.
When $\gamma_j > 0,$  there exists a $p \times \gamma_j$ full-rank matrix $\mGamma_j$ such that
\begin{align*}
  \mSigma_\ast^{-1/2} \mTheta_\ast^\top  \mX_\ast^\top  (\mI_n - \mP_j) \mX_\ast \mTheta_\ast \mSigma_\ast^{-1/2}
  =
  \mGamma_j\mGamma_j^\top ,
\end{align*}
because the noncentrality matrix is symmetric.
Note that since $\gamma_j \leq p,$ we have $\rank \mGamma_j = \gamma_j.$
Besides, $\mGamma_j^\top \mGamma_j$ is a symmetric and positive definite matrix, so that there exists a $\gamma_j \times \gamma_j$ invertible matrix $\mDelta_j$ such that $\mDelta_j^2 = \mGamma_j^\top \mGamma_j.$
Finally, $\lambda_j$ denotes the smallest eigenvalue of $\mDelta_j^2;$ that is, $\lambda_j=\lambda_{\text{min}}(\mDelta_j^2).$

To show the consistency of our information criteria, we extend the results given by \cite{yanagihara2015consistency} under (\ref{ass:unif_asy}).
They provided some sufficient conditions of the consistency of information criteria with the form of
\begin{align} \label{eq:ic_family}
  {\rm IC}_m(j)
  =
  n\log|\hat \mSigma_j| + np(\log2\pi + 1) + m(j)
  .
\end{align}
Many information criteria, such as the AIC, the exact AIC, and the BIC, belong to this class.

We introduce the following assumptions:
\begin{assumption} \label{ass:cons_ass}
  \mbox{}
  \begin{enumerate}
    \renewcommand{\labelenumi}{(\arabic{enumi})}
    \item $j_\ast \in \cJ;$ \label{ass:true_model}
    \item For $j \in \cJ_-,$  $\gamma_j$ is constant, and $\liminflh (np)^{-1}\lambda_j>0,$ \label{ass:ncm}
  \end{enumerate}
\end{assumption}
\noindent
which were essentially assumed in \cite{yanagihara2015consistency} as well.
The first assumption is trivial since we are focusing on consistency properties of variable selection.
If the second one fails, this means that the data is too uninformative because some information $\mX_\ast$ and $\mTheta_\ast$ have is asymptotically oppressed by the noise.
\cite{yanagihara2015consistency} also provide specific conditions for a two-way MANOVA model with a certain structure and claim that this assumption is realistic.

The theorem below characterizes the consistency properties of information criteria that belong to (\ref{eq:ic_family}).
The statement when $p$ is unbounded is exactly the same as Theorem 3.2 of \cite{yanagihara2015consistency}, and that for bounded $p$ is a direct extension of their Theorem 3.1.
In their study, it is assumed that $p$ is constant, but here we relax this assumption.
The proof is given in the appendix.

\begin{theorem} \label{thm:cons_ic}
  Assume Assumption \ref{ass:cons_ass}.
  Then $IC_m$ is consistent in (\ref{ass:unif_asy}) if either of the following two conditions holds:
  \begin{enumerate}
    \renewcommand{\labelenumi}{(\arabic{enumi})}
    \item $p$ is unbounded, (HD-1) $\limlh (n \log p)^{-1}(m(j) - m(j_\ast)) > - \gamma_j$ for all $j \in \cJ_-,$ and (HD-2) $\limlh p^{-1}(m(j) - m(j_\ast)) > - c_0^{-1} (k_j - k_\ast) \log(1 - c_0)$ for all $j \in \cJ_+ \setminus \{j_\ast\};$
    \label{cond:hd}
    \item $p$ is bounded, (LS-1) $\limlh n^{-1}(m(j) - m(j_\ast)) \geq 0$ for all $j \in \cJ_-,$ and (LS-2) $\limlh (m(j) - m(j_\ast)) = \infty$ for all $j \in \cJ_+ \setminus \{j_\ast\}.$
    \label{cond:ls}
  \end{enumerate}
\end{theorem}

Here we regard $c_0^{-1} \log (1 - c_0) = -1$ when $c_0 = 0.$

By definition of $\text{MPIC}_\pi,$ it belongs to the family (\ref{eq:ic_family}), and so does $\text{MPIC}_{\text{Approx}}.$
Hence, we can directly use the theorem above to show their consistency properties.
To show them, we introduce the following additional assumption:
\begin{assumption} \label{ass:x}
  For all $j \in \cJ,$
  \begin{align*}
    \liminflh \left|\frac{\mX_j^\top \mX_j}{n}\right| > 0
    .
  \end{align*}
\end{assumption}
It may be time-consuming to confirm if this assumption is satisfied for all $j \in \cJ.$
However, it is remarkable that the condition
\begin{align*}
  \liminflh \lambda_{\text{min}}\left(\frac{\mX^\top \mX}{n}\right) > 0
\end{align*}
is sufficient for Assumption \ref{ass:x}, and that this is much easier to check.
As is discussed above, many studies of linear regression models assume this.

From Lemma \ref{lem:bf}, it follows that
\begin{align*}
  m_\pi(j) - m_\pi(j_\ast)
  =
  \frac{np(2n - p - 1)(k_j - k_\ast)}{(n - p - k_j - 1)(n - p - k_\ast - 1)} + 2 \log\frac{w_\ast}{w_j}
  +
  o_p(1)
  ,
\end{align*}
and
\begin{align*}
  m_{\text{Approx}}(j) - m_{\text{Approx}}(j_\ast)
  =
  \frac{np(2n - p - 1)(k_j - k_\ast)}{(n - p - k_j - 1)(n - p - k_\ast - 1)} + 2 \log\frac{w_\ast}{w_j}
  .
\end{align*}
It is clear that $\text{MPIC}_\pi$ and $\text{MPIC}_{\text{Approx}}$ satisfy the assumptions given in the theorem above if $w_j$ is properly set.
Consequently, on their consistency properties, we have the following theorem.
We omit the proof since this follows immediately from Theorem \ref{thm:cons_ic}.

\begin{theorem} \label{thm:cons_mpic}
  Assume Assumption \ref{ass:pi}, Assumption \ref{ass:cons_ass}, and Assumption \ref{ass:x}.
  Then $\text{MPIC}_\pi$ and $\text{MPIC}_{\text{Approx}}$ are consistent in (\ref{ass:unif_asy}) if either of the following two conditions holds:
  \begin{enumerate}
    \renewcommand{\labelenumi}{(\arabic{enumi})}
    \item $p$ is unbounded, (HD-1') $\limlh (n \log p)^{-1} \log (w_\ast /  w_j)>- \gamma_j / 2$ for all $j \in \cJ_-,$ and (HD-2') $\limlh p^{-1} \log (w_\ast /  w_j) > - 2^{-1} (k_j - k_\ast) \left\{c_0^{-1}\log(1 - c_0) + (2 - c_0)(1 - c_0)^{-2}\right\}$ for all $j \in \cJ_+ \setminus \{j_\ast\};$
    \item $p$ is bounded, (LS-1') $\limlh n^{-1} \log (w_\ast /  w_j) \geq 0$ for all $j \in \cJ_-,$ and (LS-2') $\limlh \log (w_\ast /  w_j) = \infty$ for all $j \in \cJ_+ \setminus \{j_\ast\}.$
  \end{enumerate}
\end{theorem}

\section{Discussion} \label{sec:discussion}
Recall that $w_j$ lies in $(0, 1)$ and is a parameter that determines whether the data is high-dimensional or not.
To use the $\text{MPIC}_\pi$ in practical applications, one must determine $w_j$ so that it satisfies the conditions of the consistency.
In this section, we propose some ways to determine $w_j.$
One simple example of $w_j$ is
\begin{align} \label{eq:simple_w}
  w_j^\prime
  =
  n^{-\varepsilon k_j}
\end{align}
for some positive $\varepsilon.$
It is easy to show that this satisfies all the assumptions provided above.

It is also possible to set a prior for $w_j$ and take its posterior mean as the weight so that it depends on the data as well as $n, p,$ and $k_j.$
For example, assume that $w_j \sim \text{Beta}(\alpha_j, \beta_j)$ for $\alpha_j, \beta_j > 0.$
Then the posterior mean of $w_j$ in the model (\ref{eq:ss_model}) is
\begin{align*}
  w_j^{\prime\prime}
  =
  E[w_j \mid \mY]
  =
  \frac{\alpha_j}{\alpha_j + \beta_j + 1}
  \left(1 + \frac{f(\mY \mid \mX_j \hat \mTheta_j, \hat\mSigma_j)}{\beta_j f_\pi(\mY \mid \hat \mSigma_j) + \alpha_j f(\mY \mid \mX_j \hat \mTheta_j, \hat\mSigma_j)}\right)
  .
\end{align*}
Consider the case of $\beta_j = n^{\varepsilon k_j}$ and $\liminflh \alpha_j > 0.$
It is easy to show that
\begin{align*}
  \frac{\beta_j f_\pi(\mY \mid \hat \mSigma_j)}{f(\mY \mid \mX_j \hat \mTheta_j, \hat\mSigma_j)}
  =
  o_p(1)
  ,
\end{align*}
and so
\begin{align*}
  w_j^{\prime\prime}
  =
  \frac{\alpha_j + 1}{\alpha_j + \beta_j + 1} + o_p(1)
\end{align*}
for a broad class of smooth distributions.
When $\alpha_j$ is constant, $w_j^{\prime\prime}$ is asymptotically equivalent to (\ref{eq:simple_w}) and induces a consistent variable selection method.
When $\alpha_j = p^{\varepsilon k_j},$ the weight above is approximated by (\ref{eq:w}).
Because this weight depends on $p,$ one can regard it as a measure of the high-dimensionality of the data.
Indeed, it is almost zero when $n$ is large enough, and it is away from zero if $p$ is large and its order is comparable to that of $n.$

We currently have no way to set an optimal $w_j$ in general, but here we focus on the weight with the form of (\ref{eq:w}) and provide a reasonable way to choose $\varepsilon$ in (\ref{eq:w}).
As we discussed in Section \ref{sec:new_information_criterion}, when we take (\ref{eq:w}) for $\varepsilon \in (0, 1/2),$ the multivariate scale invariance and the uniform boundedness of $\pi$ are sufficient for Assumption \ref{ass:pi}.
In other words, $\varepsilon$ should be small so that all scale-invariant and uniformly bounded distributions satisfy Assumption \ref{ass:pi}; otherwise, the existence of $\ell_{n, p}$ can fail for small $p.$
However, too small $\varepsilon$ is not desirable in terms of consistency for the following reason.
As we can see in the proof of Theorem \ref{thm:cons_ic}, which is provided in the \textcolor{black}{appendix}, the asymptotic behavior of $\text{MPIC}_\pi(j) - \text{MPIC}_\pi(j_\ast)$ determines whether MPIC is consistent, and $\log (w_\ast / w_j)$ should be large to make the selection probability of the true model high because the difference increases as the logarithm grows.
Here we have $\log (w_\ast / w_j) = \varepsilon (k_j - k_\ast) \log(n / p) + O(1),$ and so when $j \in \cJ_-,$ it does not affect the consistency property of the MPIC because $\text{MPIC}_\pi(j) - \text{MPIC}_\pi(j_\ast) = O_p(n \log p),$ which is much larger than $\log (w_\ast / w_j).$
On the other hand, in the case of $j \in \cJ_+ \setminus \{j_\ast\},$ we can see that the difference between the information criterion of the true model and that of a candidate model is $O_p(p \vee \log (w_\ast / w_j)).$ That is, $\log (w_\ast / w_j)$ is asymptotically effective and should be large for a good variable seletion.
Hence, in terms of consistency, large $\varepsilon$ is preferable.
To sum up, when we use (\ref{eq:w}) as the weight, $\varepsilon$ should be as large as possible under the condition that it is less than $1/2.$

\section{Numerical Studies} \label{sec:numerical_study}
\subsection{Simulations}
In this section, we compare the probabilities of selecting the true model by the AIC, the $\text{AIC}_\text{C}$ (the exact AIC), the BIC, the GIC (\ref{eq:gic}), and the $\text{MPIC}_{\text{Approx}}.$
The same experiments as below for the $\text{MPIC}_\pi$ for $\pi(\mTheta_j \mid \mSigma) = f(\mTheta_j \mid \mO_{k_j \times p}, \mSigma)$ (a normal distribution) and $\pi(\mTheta_j \mid \mSigma) = |\mSigma|^{-k_j/2}$ (a scaled uniform distribution) give similar performances to those of the $\text{MPIC}_{\text{Approx}},$ so we omit them.
We set $\beta = (\log n) p^{-1/2}$ for the GIC and use (\ref{eq:w}) with $\varepsilon = 0.499,$ which is almost a half, as the weight for the $\text{MPIC}_{\text{Approx}}.$
We evaluate these information criteria by Monte Carlo simulations based on $1000$ replications under different values of $n$ and $p.$
We generate an $n \times (k-1)$ matrix \textcolor{black}{independently and identically drawn from a uniform distribution $U(-2, 2).$}
Concatenating the $n$-dimensional vector $\bm{1}_n$ of ones with this matrix, we construct an $n \times k$ matrix $\mX$ of explanatory variables.
The true model is determined by $j_\ast = \{1, 2, 3, 4, 5\},$ \textcolor{black}{$\mTheta_\ast = (5, 4, 3, 2, 1)^\top \bm{1}_p,$} and $\mSigma_\ast = 0.8\mI_p + 0.2\bm{1}_p\bm{1}_p^\top .$
These are common settings in numerical studies of variable selection of the linear regression model (e.g., \cite{bedrick1994model}).

In the first experiment, we consider the case when the set of candidate models is nested; that is, $\cJ = \{j_1, \cdots, j_{10}\},$ where $j_\alpha = \{1, \cdots, \alpha\}.$
The results are given in Figure \ref{fig:prob_nested}.
First of all, we can see that the AIC and the exact AIC select the true model nearly with probability one when both the data size and the dimension are large, and that their selection probabilities do not attain one when only the data size is large, to the contrary.
In the cases of the BIC, the selection probabilities get higher as the data size gets larger with the dimension fixed, but when the dimension is also large, it does not work at all.
Our new information criterion and the GIC seem consistent in both LS and HD cases, but the performances of the MPIC are better than those of the GIC especially when the data size is not very large.
%Although the $\text{MPIC}_{\text{Approx}}$ is derived by an approximation, its selection probabilities are not deteriorated compared with those of the $\text{MPIC}_{\text{Normal}}$ and the $\text{MPIC}_{\text{Uniform}}.$

\begin{figure}[h]
\centering
\includegraphics[width=14cm]{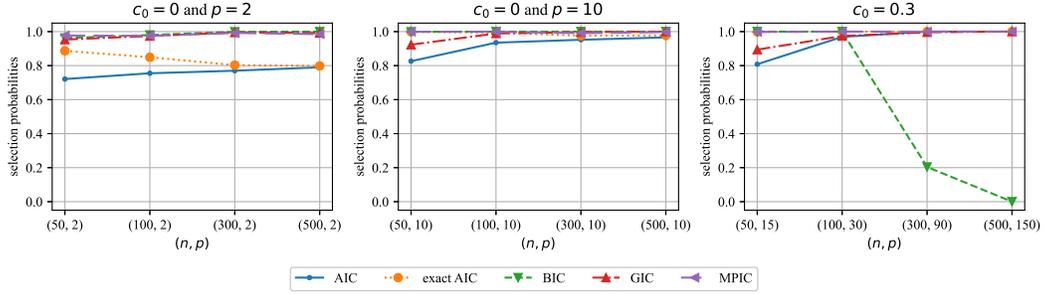}
\caption{\label{fig:prob_nested} Selection probabilities for nested models}
\end{figure}

\textcolor{black}{Next, we consider the case when the set of candidate models is non-nested; that is, $\cJ = \{\{1\} \cup j \mid j \subset \omega \setminus \{1\}\},$ where $\omega = \{1, \cdots, 8\}.$
This specification implies that all the candidate models have a constant term.}
The results are given in Figure \ref{fig:prob_nonnested}.
As in the first setting above, we can see that the MPIC selects the true model consistently.
Compared with the GIC, the MPIC performs well especially when $n$ is not very large.

\begin{figure}[h]
\centering
%\label{fig:prob_nonnested}
\includegraphics[width=14cm]{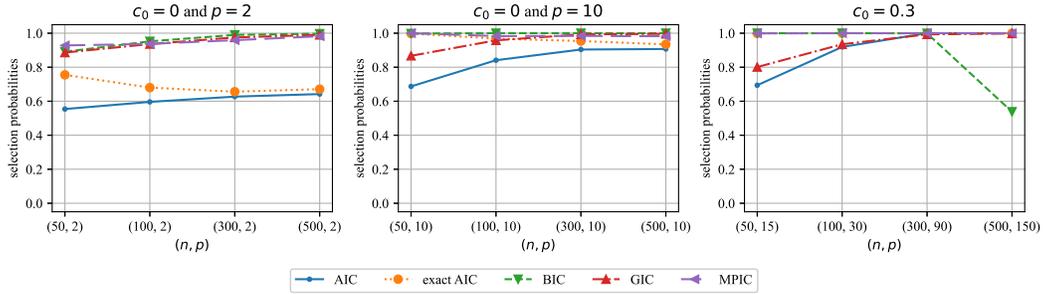}
\caption{\label{fig:prob_nonnested} Selection probabilities for non-nested models}
\end{figure}

Aside from consistency, efficiency of model selection methods has been discussed in previous studies (e.g., \cite{shibata1980asymptotically}, \cite{shibata1981optimal}, \cite{shibata1983asymptotic}, and \cite{shao1997asymptotic}).
In particular, \cite{yang2005can} shows that any consistent model selection criterion is not minimax-rate optimal in terms of prediction; that is, it is not asymptotically efficient in the LS asymptotic framework for some regression models.
This study implies that the information criteria proposed in this paper might not have the asymptotic efficiency even in the linear regression model we consider above.
The next experiment allows us to examine how inefficient our selection method is in this setup.
As a measure of prediction error, we adopt the squared loss function following \cite{shibata1983asymptotic}:
\begin{align*}
    L(\hat j)
    =
    ||\mX_\ast \mTheta_\ast - \mX_{\hat j} \hat \mTheta_{\hat j}||^2
    ,
\end{align*}
and consider the ratio to the risk for the true model:
\begin{align*}
    \text{eff}(\hat j)
    =
    \frac{E[L(\hat j)]}{E[L(j_\ast)]}
    .
\end{align*}
A selection $\hat j$ is said to be asymptotically (mean) efficient if $\liminflh \text{eff}(\hat j) = 1.$
The ratio for each information criterion is reported in Table \ref{tab:efficiency}.
According to this, the MPIC is as efficient as the AIC not only in the LS but also in the HD framework.
As \cite{yang2005can} points out, it is possible that the MPIC does not have asymptotic efficiency for some models, but this does not seem to be problematic for linear models we address in this paper.
\begin{table}[htb]
\begin{center}
%\hspace*{-0.8cm}
\caption{\label{tab:efficiency} Risk ratio for nested models}
\vspace{0.2cm}
\begin{tabular}{ccccccc}
\toprule
$n$  & $p$ & AIC   & $\text{AIC}_{\text{C}}$ & BIC   & GIC   & $\text{MPIC}_{\text{Approx}}$ \\
\midrule
50  & 2   &  1.092442 &  1.027110 &     1.010272 &  1.010875 &           1.005960 
\\
100 & 2   &  1.066522 &  1.039487 &     1.004284 &  1.007055 &           1.006684 
\\ 
300 & 2   &  1.065083 &  1.055309 &     1.003342 &  1.004176 &           1.005278 
\\
500 & 2   &  1.055427 &  1.049095 &     1.001377 &  1.002763 &           1.003802 
\\
\midrule
50  & 10  &  1.035427 &  0.998426 &     0.998426 &  1.011384 &           0.998426 
\\
100 & 10  &  1.018776 &  1.001509 &     1.000316 &  1.004279 &           1.000709 
\\
300 & 10  &  1.010330 &  1.004110 &     1.000135 &  1.000541 &           1.001543 
\\
500 & 10  &  1.010607 &  1.005468 &     0.999969 &  1.000359 &           1.000983 
\\
\midrule
50  & 15  &  1.048427 &  1.000570 &     1.000570 &  1.023782 &           1.009885 
\\
100 & 30  &  1.005265 &  0.998728 &     0.998728 &  1.004270 &           0.998728 
\\
300 & 90  &  1.000827 &  1.000827 &  1924.787382 &  1.001246 &           1.000827 
\\
500 & 150 &  1.001262 &  1.001262 &  3994.138044 &  1.001658 &           1.001262 
\\
\bottomrule
\end{tabular}
\end{center}
\end{table}

\textcolor{black}{
We now consider to what extent the non-normality affects model selection based on each information criterion.
To see this, we think about the same linear model but different error term distributions. 
The error is assumed to be generated as $\mE = \mDelta \mSigma^{1/2}$ where each element of $\mDelta$ is generated from the following four distributions: 
(i) the Laplace distribution with scale parameter $0.5$, 
(ii) the $t$-distribution with degree $4$, 
(iii) the $\chi^2$ distribution with degree $2$, and 
(iv) the standard normal distribution $0.05$-contaminated by the Cauchy distribution.
The performances given in Figure \ref{fig:prob_robust} are similar to those for the normal distribution.
The selection probabilities of the (exact) AIC are high when $p$ is large enough, and the BIC selects the true model with a high probability for LS data.
Selection probabilities of the GIC and the MPIC are high similarly.
}
\begin{figure}[h]
\centering
\includegraphics[width=14cm]{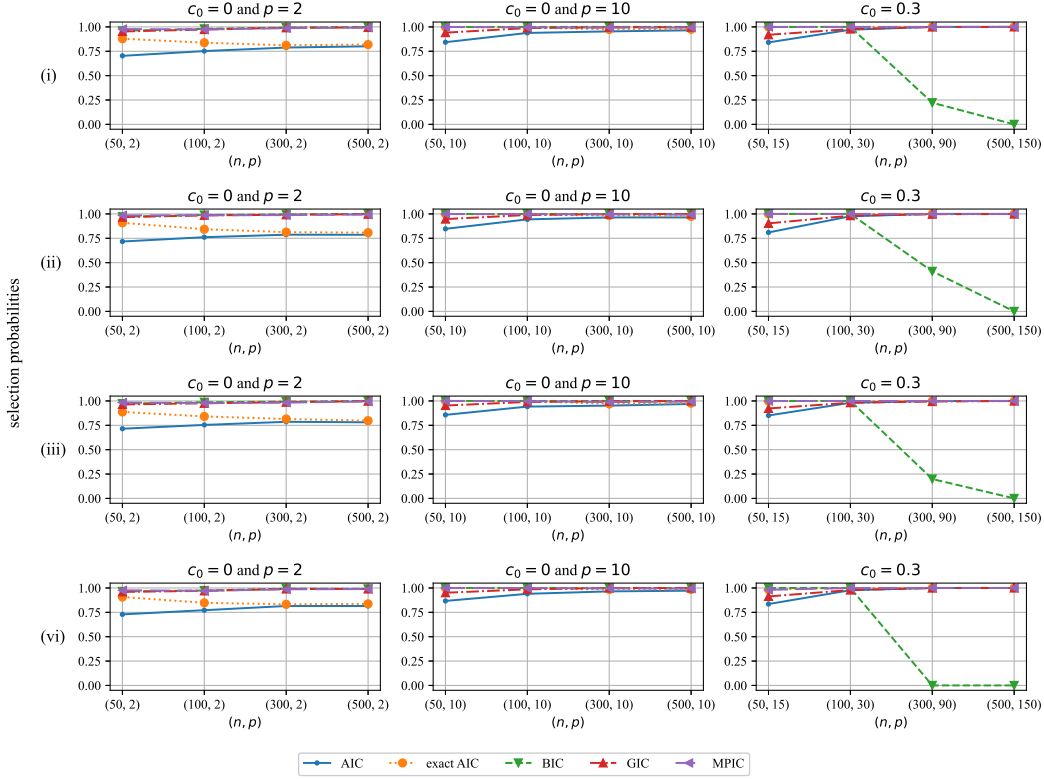}
\caption{\label{fig:prob_robust} Selection probabilities for non-normal models}
\end{figure}

We end this subsection with a numerical experiment on the choice of $w_j.$
We compare the selection probabilities by the $\text{MPIC}_{\text{Approx}}$ for (\ref{eq:w}) with $\varepsilon = 10, 5, 2.5, 1, 0.5, 0.499, 0.25, 0.1.$
It is noted that the assumptions of Theorem \ref{thm:cons_mpic} are satisfied for all of these $\varepsilon.$
However, all of the multivariate-scale invariant and uniformly bounded prior distributions satisfy Assumption \ref{ass:pi} for the last three whereas for the other four, some of such priors, including the normal distribution and the scaled uniform distribution, violate it when $p$ is small, and so the derivation of the MPIC is not justified in the way we propose in Section \ref{sec:new_information_criterion}.
Figure \ref{fig:epsilon_comparison} shows probabilities of selecting the true model out of the nested set of candidate models under the same setting as in the first experiment.
As one can see in the figure, all criteria select the true model with high probability when $n$ is sufficiently large, which is consistent with the statement of Theorem \ref{thm:cons_mpic}.
As we discuss in Section \ref{sec:discussion}, for $\varepsilon < 1/2,$ the larger $\varepsilon$ gets, the better performances the criterion exhibits especially when $n$ and $p$ are small.
Even when $\varepsilon$ exceeds half, the selection probabilities are not very deteriorated, or rather they are better in many cases.
%In particular, for $\varepsilon = 2.5,$ the MPIC selects the true model with a very high probability in all cases provided in the table.
For too large $\varepsilon,$ say $\varepsilon = 10,$ the performances are unstable when $n = 50.$
This is because the term of $O(1)$ is not ignorable and has a large impact on variable selection when $n$ is small.
These experiments imply the possibility that Assumption \ref{ass:pi} is not necessary for the derivation of the MPIC, and that the optimal $\varepsilon$ in variable selection may be greater than $1/2.$

\begin{figure}[h]
\centering
\includegraphics[width=14cm]{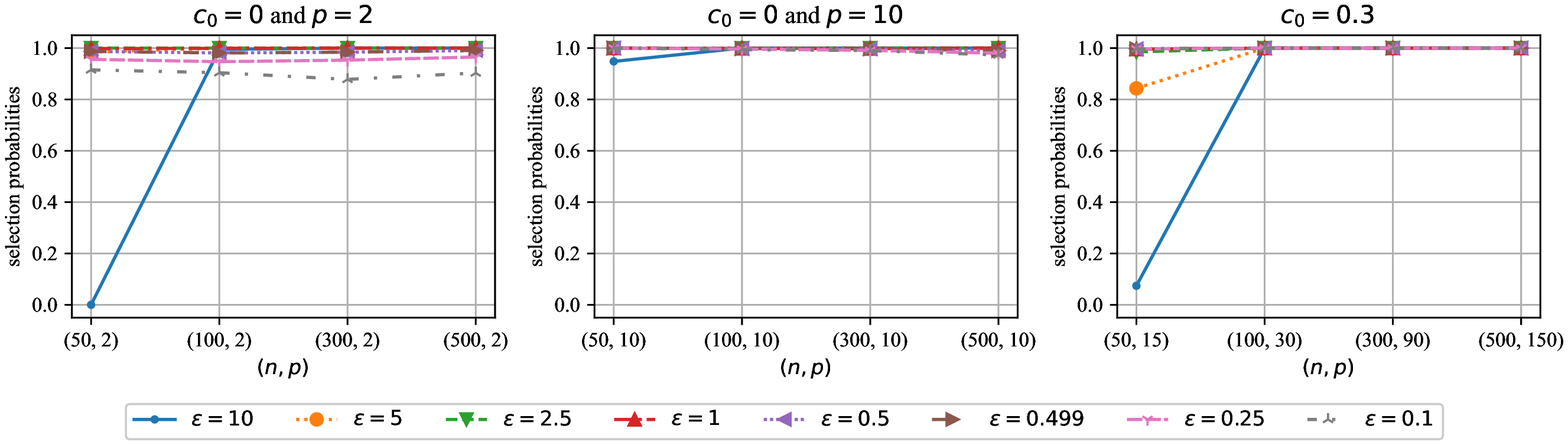}
\caption{\label{fig:epsilon_comparison} Selection probabilities for different $\varepsilon$}
\end{figure}

\subsection{\textcolor{black}{Illustrative Example}}
In this subsection, we employ model selection methods to a factor asset pricing model, which plays a significant role in modern financial economics (e.g., \cite{fama1993common} and \cite{fama2015five}).
We consider the following model:
\begin{align*}
    R_{t}
    =
    \mTheta^\top X_t
    +
    \varepsilon_t
    ,
    \ \
    \varepsilon_t \sim N_p(0, \mSigma)
    ,
\end{align*}
where $R_t$ is a $p$-dimensional vector of excess returns of portfolio, $X_t$ is a $k$-dimensional vector of factor portfolio excess returns, and $\mTheta$ is a $k \times p$ matrix of coefficients.
Furthermore, we allow each entry of the error term to exhibit an AR(1) structure.
That is, $\varepsilon_t$ follows
\begin{align*}
    \varepsilon_{t, i}
    =
    \rho \varepsilon_{t - 1, i}
    +
    \omega_{t, i}
    ,
\end{align*}
where $\omega_t \sim N_p(0, 1).$
Note that the variance of $\omega_{t,i}$ can be assumed to be one without loss of generality because the scale parameter can be incorporated into $\mSigma.$

The model we consider is different from what we have investigated so far in that $\{\varepsilon_t\}$ is not independent.
To modify our basic model, consider the following multivariate linear regression model for non-IID observations:
\begin{align*}
    \mY
    \sim
    N_{n \times p}(\mX_j \mTheta_j, \mSigma_j \otimes \mV)
    ,
\end{align*}
where $\mV$ is an $n \times n$ covariance matrix that captures the autocorrelation structure of data.
If the true value of $\mV$ is known, this model can be boiled down to the standard model by a simple transformation:
\begin{align*}
    \mV^{-1/2} \mY
    \sim
    N_{n \times p}(\mV^{-1/2}\mX_j \mTheta_j, \mSigma_j \otimes \mI_n)
    .
\end{align*}
Provided this fact, we can reasonably apply model selection methods for the standard model to models with autocorrelation by estimating $\mV$ before the model selection.
Namely, we follow the steps below: we first estimate $\rho$ with the OLS $\hat \rho$ from all the available data; then for the estimated covariance matrix $\hat \mV = \{\hat \rho^{|i-j|}\}_{i,j},$ we execute model selection against the transformed data $({\hat \mV}^{-1/2} \mY, {\hat \mV}^{-1/2} \mX).$

As the response variables, we use monthly returns of 32 Japanese portfolios formed on size, operating profitability, and investment.
The explanatory variables include a constant, Fama-French's Japanese five factors (market excess return, SMB, HML, RMW, and CMA) proposed in \cite{fama2015five}, Japanese momentum factor (WML, cf., \cite{carhart1997persistence}), Nikkei volatility index (JNIV), which is a volatility index of Japanese stocks, and daily cases of Covid 19 in Japan.
In addition to the standard economic variables in the context of portfolio selection, we also adopt the last two as regressors because some studies claim that the volatility index and confirmed cases have an impact on the stock market (see, e.g., \cite{ang2006cross} and \cite{ashraf2020stock}).
We use the data from November, 2014 to October, 2021.
The set of candidate models is all subsets that contains a constant term, and consequently, the size of the set is $2^8 = 256.$

In this experiment, we observe two points: (i) which criterion selects which variables and (ii) which criterion has a good prediction precision.
For these purposes, we split the whole data into two parts, before and after October, 2020.
Then we first do the model selection and coefficient estimation using the first part, and after that, evaluate the prediction error with the rest.
We can specify the (training) data dimension as follows: $n = 72, p = 32,$ and $k = 9.$
As for (ii), this setup conflicts with the standard linear model we address in this paper in that $\mX,$ as well as $\mE,$ should be treated as stochastic, so theoretical results on asymptotic efficiency of model selection which we have reviewed in the previous subsection do not necessarily hold in this case.

The results are summarized in Table \ref{tab:stock}.
The AIC and the GIC select relatively overspecified models because they have weaker penalty terms.
This property implies that their predicted values are likely to be vulnerable to explanatory variables that have weak correlation with the response variables.
Indeed, their prediction errors are huge because they put too much value on daily cases of Covid 19, which fluctuate intensely in the test data.
On the other hand, the exact AIC, the BIC, and the MPIC select fewer explanatory variables and seem to stably predict the future values without being affected by unnecessary explanatory variables.
\begin{table}
\begin{center}
\caption{\label{tab:stock} Selected variables and prediction errors (the columns represent all the explanatory variables: constant term, market excess return, SMB, HML, RMW, CMA, WML, JNIV, and daily cases of Covid 19, respectively)}
\vspace{0.2cm}
\begin{tabular}{c|ccccccccc||c}
\toprule
 &  0 &  1 &  2 &  3 &  4 &  5 &  6 &  7 &  8 & prediction error
\\
\midrule
AIC &  $\checkmark$ & $\checkmark$ & $\checkmark$ & $\checkmark$ & $\checkmark$ & $\checkmark$ & & & $\checkmark$  & 57.62799
\\
$\text{AIC}_{\text{C}}$ &  $\checkmark$ &  $\checkmark$ &  $\checkmark$ &  &  $\checkmark$ &  $\checkmark$ &  &  &  & 3.890582
\\
BIC &  $\checkmark$ & $\checkmark$ & $\checkmark$ & $\checkmark$ & $\checkmark$ & $\checkmark$ &  &  &  & 3.80521
\\
GIC & $\checkmark$ & $\checkmark$ & $\checkmark$ & $\checkmark$ & $\checkmark$ & $\checkmark$ &  & & $\checkmark$ & 57.62799
\\
$\text{MPIC}_{\text{Approx}}$ & $\checkmark$ & $\checkmark$ & $\checkmark$ &  & $\checkmark$ & $\checkmark$ &  &  & & 3.890582
\\
\bottomrule
\end{tabular}
\end{center}
\end{table}

\section{Concluding Remarks} \label{sec:conclusion}
In this paper, we derived several variable selection criteria for the multivariate normal linear regression model based on the Bayesian mixture marginal likelihood.
We also demonstrated that they are consistent in an asymptotic framework that contains the LS and HD frameworks.
Unlike many studies that artificially adjusted the penalty term of the AIC, our methods of model selection stand on the concept of minimizing the KL risk, and they are likely to be extended to other models due to their intuitively comprehensible properties.
From this theoretical perspective, we insist that the information criteria we provide in this paper should be favorable to existing model selection methods like the GIC even though their performances are similar.

We complete this paper with some discussion on its limitations and possible directions of further research.
First, we cannot apply our proof of consistency to non-normal models as they stand, although the basic idea of the derivation of our information criteria remains valid.
We expect that it is possible that our results can be extended to elliptical distribution models by using their properties similar to the normal distribution.

It is also noted that we do not consider a prior on the whole set of possible models.
In other words, this means that the uniform distribution is set over models as a natural noninformative choice.
An interesting extension of the results given in this paper is to consider the Bayesian model selection in the framework that incorporates the model space prior following \cite{clyde2004model}.
For example, the independent Bernoulli prior $\pi(j \mid \gamma) =\gamma^{k_j}(1-\gamma)^{k_\omega - k_j}$ is a default choice, where $j$ is the index of a model, and $\gamma$ is a single hyperparameter $\gamma \in (0, 1)$.

In most of the literature on the multivariate linear regression model, the following four asymptotic frameworks have been considered: (a) $n\to\infty$ and $p$ and $k$ bounded, (b) $(n, p)\to\infty$ and $k$ bounded, (c) $(n, k)\to\infty$ and $p$ bounded, and (d) $(n, p, k)\to\infty$.
The asymptotics (a) is the large sample case and the others are high-dimensional cases.
This paper covers (a) and (b) under the constraint $p < n - k - 1$, which is required for $E[\hat \mSigma_j^{-1}]$ to exist.
This constraint may be relaxed in the derivation of the MPIC by replacing $\hat \mSigma_j$ with a ridge-type estimator $\hat \mSigma_j^{\rm R}= n^{-1}\{\mY^\top(\mI_n-\mP_j)\mY+ \hat \lambda \mI_p\}$ for an appropriate positive statistic $\hat \lambda.$
However, it is technically difficult to show the consistency in the case of $p>n$.
In the asymptotics (c), \cite{chen2008extended} suggested the extended Bayesian information criterion by introducing a penalty term that is related to the cardinality of the collection of candidate models and showed the consistency in the high-dimensional asymptotic framework (c) with $k=O(n^\gamma)$ for some constant $\gamma > 0.$
Although they considered the univariate linear regression model, their results could be extended to the multivariate model.
An interesting question is whether the results in this paper can be extended to the asymptotic case (d) with $p=O(n)$ and $k=O(n^\gamma)$ under the constraint $p <n - k - 1,$ which we leave for future research.

Although our main focus of this paper is the Bayesian information criterion, the penalized least squares methods are popular and useful for selecting variables in the high dimension.
\cite{zou2009adaptive} provided the adaptive elastic net under the $\ell_1$-penalty in the univariate linear regression model and showed the consistency in the high-dimensional case (c) with $p=1.$
In the setup of the multivariate linear regression model, \cite{chen2012sparse} suggested the sparse reduced-rank regression method through penalized least squares and showed the consistency in the large sample case (a).
\cite{li2015multivariate} derived the multivariate sparse group lasso and obtained the prediction and estimation error bounds, but they did not investigate the consistency.
%Thus, the consistency in the high-dimensional cases are desirable in the multivariate linear regression model.
An advantage of the suggested MPIC over these penalized least squares methods is that its consistency is guaranteed in the high-dimensional case (b).

%Our assumption that the number $k$ of explanatory variables is fixed is another problem in practical applications.
%For example, when one considers the order selection problem of the vector autoregressive model, $k$ grows as $p$ does.
%To overcome this obstacle, we have to investigate the properties of our criteria in another asymptotic framework that allows $k$ to vary, which we leave for future research.

\subsection*{ACKNOWLEDGMENTS}
We would like to thank the associate editor and four reviewers for many valuable comments and helpful suggestions that led to an improved version of this paper.

\printbibliography{ref}

\appendix
\section{Proofs}

\subsection{Proof of Theorem \ref{thm:bias}} \label{subsec:thm:bias}
We first provide a useful lemma.
\begin{lemma} \label{lem:e_log}
  Let $\{a_{n, p}\}$ be a sequence of positive real numbers that satisfies $a_{n,p} \exp(up) = o(1)$ for any $u \geq 0,$ and suppose that a sequence $\{Y_{n, p}\}$ of random variables is $L^2$-bounded.
  Then for $Z_{n, p} = a_{n, p} \exp(p Y_{n, p}),$
  \begin{align} \label{eq:exp_log}
    E[\log (1 + Z_{n, p})] = o(p)
  \end{align}
  holds.
\end{lemma}

\begin{proof}
  The left-hand side of (\ref{eq:exp_log}) can be represented as
  \begin{align*}
    E[\log (1 + Z_{n, p})]
    =
    E
    \left[
      \int_0^{Z_{n, p}}
      \frac{1}{1+t}
      dt
    \right]
    =
    \int_0^\infty
      \pr(Z_{n, p} \geq t)
      \frac{1}{1 + t}
    dt
    .
  \end{align*}
  Note that we use Fubini's theorem in the second equality.
  By the change of variables, we have
  \begin{align*}
    \int_0^\infty
      \pr(Z_{n, p} \geq t)
      \frac{1}{1 + t}
    dt
    =
    \int_0^1
      \pr(\exp(pY_{n, p}) \geq s)
      \frac{a_{n, p}}{1 + a_{n, p}s}
    ds
    +
    \int_1^\infty
      \pr(\exp(pY_{n, p}) \geq s)
      \frac{a_{n, p}}{1 + a_{n, p}s}
    ds
    .
  \end{align*}
  The first term can be evaluated as
  \begin{align*}
    0
    \leq
    \int_0^1
      \pr(\exp(pY_{n, p}) \geq s)
      \frac{a_{n, p}}{1 + a_{n, p}s}
    ds
    \leq
    a_{n, p}
    \to
    0
    .
  \end{align*}
  For the second term, by the change of variables again, we have
  \begin{align*}
    \int_1^\infty
      \pr(\exp(pY_{n, p}) \geq s)
      \frac{a_{n, p}}{1 + a_{n, p}s}
    ds
    =
    \int_0^1
      \pr(Y_{n, p} \geq u) \frac{a_{n, p} p \exp(up)}{1 + a_{n, p} \exp(up)}
    du
    +
    \int_1^\infty
      \pr(Y_{n, p} \geq u) \frac{a_{n, p} p \exp(up)}{1 + a_{n, p} \exp(up)}
    du
    .
  \end{align*}
  The first term asymptotically vanishes because
  \begin{align*}
    \int_0^1
      \pr(Y_{n, p} \geq u) \frac{a_{n, p} p \exp(up)}{1 + a_{n, p} \exp(up)}
    du
    \leq
    \int_0^1
      \frac{a_{n, p} p \exp(up)}{1 + a_{n, p} \exp(up)}
    du
    =
    \log \frac{1 + a_{n, p}\exp(up)}{1 + a_{n, p}}
    \to
    0
    .
  \end{align*}
  On the other hand, the second term can be evaluated by using Markov's inequality in the following way:
  \begin{align*}
    \int_1^\infty
      \pr(Y_{n, p} \geq u) \frac{a_{n, p} p \exp(up)}{1 + a_{n, p} \exp(up)}
    du
    \leq
    p
    \left(\sup_{n, p} E [Y_{n, p}^2]\right)
    \int_1^\infty
      \frac{1}{u^2}
      \frac{a_{n, p} \exp(up)}{1 + a_{n, p} \exp(up)}
    du
    .
  \end{align*}
  From the assumption, we have $\sup_{n, p} E Y_{n, p}^2 < \infty$.
  We can apply Lebesgue's convergence theorem to the integral of the right-hand side, and we can see that it goes to zero.
  From these facts, we have
  \begin{align*}
    \int_1^\infty
      \pr(Y_{n, p} \geq u) \frac{a_{n, p} p \exp(up)}{1 + a_{n, p} \exp(up)}
    du
    =
    o(p)
    .
  \end{align*}
  Combining all results above yields (\ref{eq:exp_log}).
\end{proof}

Now, we move on to the proof of Theorem \ref{thm:bias}.
The bias term can be written as
\begin{align} \label{eq:bias}
  b_j
  &=
  E^{\mY, \tilde\mY} \left[-2 \log f(\tilde\mY \mid \mX_j \hat \mTheta_j, \hat \mSigma_j) + 2 \log f(\mY \mid \mX_j \hat \mTheta_j, \hat \mSigma_j)\right]
  \nonumber
  \\
  & \ \ \ \ \ \ +
  E^{\mY, \tilde\mY}\left[
    -2
    \log \left\{
      1
      +
      \frac{(1 - w_j) f_\pi(\tilde \mY \mid \hat\mSigma_j)}{w_j f(\tilde \mY \mid \mX_j \hat\mTheta_j, \hat\mSigma_j)}
    \right\}
    +
    2
    \log \left\{
      1
      +
      \frac{(1 - w_j) f_\pi(\mY \mid \hat\mSigma_j)}{w_j f(\mY \mid \mX_j \hat\mTheta_j, \hat\mSigma_j)}
    \right\}
  \right]
  .
\end{align}
Since the first term is the bias that appears in the derivation of the exact AIC, we have
\begin{align} \label{eq:bias_aicc}
  E^{\mY, \tilde\mY} \left[-2 \log f(\tilde\mY \mid \mX_j \hat \mTheta_j, \hat \mSigma_j) + 2 \log f(\mY \mid \mX_j \hat \mTheta_j, \hat \mSigma_j)\right]
  =
  \frac{np(2k_j + p + 1)}{n - p - k_j - 1}
  .
\end{align}
Then what remains to be shown is that the second term of (\ref{eq:bias}) is $o(p)$.
From Assumption \ref{ass:pi},
\begin{align*}
  \int f(\tilde \mY \mid \mX_j\mTheta_j, \hat \mSigma_j) \pi(\mTheta_j \mid \hat \mSigma_j) d\mTheta_j
  &=
  \int f(\tilde \mY \mid \mX_j \mXi_j \hat \mSigma_j^{1/2}, \hat \mSigma_j) \pi(\mXi_j \mid \mI_p) d\mXi_j
  \\
  &\leq
  \pi_M
  \int f(\tilde \mY \mid \mX_j \mTheta_j, \hat \mSigma_j) d\mTheta_j
  |\hat \mSigma_j|^{-k_j/2}
  \\
  &=
  \pi_M
  \frac{
    \etr(-(1/2) \hat \mSigma_j^{-1} \tilde \mY^\top  (\mI_n - \mP_j) \tilde \mY)
  }
  {(2 \pi)^{(n - k_j)p/2} |\hat \mSigma_j|^{n/2} |\mX_j^\top  \mX_j|^{p/2}}
\end{align*}
holds.
Therefore, we have
\begin{align*}
  \frac{(1 - w_j) f_\pi(\tilde \mY \mid \hat\mSigma_j)}{w_j f(\tilde \mY \mid \hat\mTheta_j, \hat\mSigma_j)}
  &=
  \frac{1 - w_j}{w_j}
  \frac{\int f(\tilde \mY \mid \mX_j\mTheta_j, \hat \mSigma_j) \pi(\mTheta_j \mid \hat \mSigma_j) d\mTheta_j}{f(\tilde \mY \mid \mX_j \hat \mTheta_j, \hat \mSigma_j)}
  \\
  &\leq
  \frac{1 - w_j}{w_j}
  \pi_M
  \left(
    \frac
      {2\pi}
      {\left|\frac{\mX_j^\top \mX_j}{n}\right|^{1/k_j} n}
  \right)^{pk_j/2}
  \etr\left(\frac{1}{2} \hat \mSigma_j^{-1} (\tilde \mTheta_j - \hat \mTheta_j)^\top  \mX_j^\top  \mX_j (\tilde \mTheta_j - \hat \mTheta_j)\right)
\end{align*}
where $\tilde \mTheta_j = (\mX_j^\top \mX_j)^{-1}\mX_j^\top  \tilde \mY.$
Let
\begin{align*}
  a_{n, p}
  =
  \frac{1 - w_j}{w_j}
  \pi_M
  \left(
    \frac
      {2\pi}
      {\left|\frac{\mX_j^\top \mX_j}{n}\right|^{1/k_j} n}
  \right)^{pk_j/2}
  ,
  \ \
  Y_{n, p} = \frac{1}{2p}\tr \left[\hat \mSigma_j^{-1} (\tilde \mTheta_j - \hat \mTheta_j)^\top  \mX_j^\top  \mX_j (\tilde \mTheta_j - \hat \mTheta_j)\right]
  .
\end{align*}
With $\ell_{n, p}$ in Assumption \ref{ass:pi}, we can write $a_{n, p}$ as
\begin{align*}
  a_{n, p}
  =
  \frac{(1 - w_j)\ell_{n, p}^{pk_j/2}}{w_j n^{pk_j/2}} \pi_M
  \left(
    \frac
      {2\pi}
      {\left|\frac{\mX_j^\top \mX_j}{n}\right|^{1/k_j} \ell_{n, p}}
  \right)^{pk_j/2}
  .
\end{align*}
From Assumption \ref{ass:pi}, we can see that $a_{n, p}$ satisfies the conditions of Lemma \ref{lem:e_log}.

Next, we show that $Y_{n, p}$ is tight.
First of all, we have
\begin{align*}
  \mZ_j
  =
  \left(\frac{\mX_j^\top\mX_j}{2}\right)^{1/2} (\tilde \mTheta_j - \hat \mTheta_j) \mSigma_\ast^{-1/2}
  \sim N_{k_j \times p} (\mO_{k_j \times p} , \mI_p \otimes \mI_{k_j})
  .
\end{align*}
Then it holds that
\begin{align*}
  Y_{n, p}
  =
  \frac{n}{p} \tr \left[\left(n\mSigma_\ast^{-1/2}\hat \mSigma_j\mSigma_\ast^{-1/2}\right)^{-1} \mZ_j^\top  \mZ_j\right]
  .
\end{align*}
Notice that
\begin{align*}
  \mW_j
  =
  n\mSigma_\ast^{-1/2}\hat \mSigma_j\mSigma_\ast^{-1/2}
  \sim
  W_p(n - k_j, \mI_p)
  ,
\end{align*}
and that $\mW_j$ is independent of $\mZ_j$.
Accordingly,
\begin{align*}
  Y_{n,p}^2
  =
  \left(\frac{n}{p} \tr (\mW_j^{-1} \mZ_j^\top  \mZ_j) \right)^2
  =
  \frac{n^2}{p^2} \left\{(\vecop(\mZ_j^\top ))^\top  (\mI_{k_j} \otimes \mW_j^{-1}) \vecop(\mZ_j^\top ) \right\}^2
  .
\end{align*}
(For the $\vecop$ operator, see \cite{timm2002applied}, for example.)
Since $\mZ_j$ and $\mW_j$ are independent,
\begin{align*}
  E[Y_{n, p}^2 \mid \mW_j]
  &=
  \frac{n^2}{p^2}
  E\left[\left\{(\vecop(\mZ_j^\top ))^\top  (\mI_{k_j} \otimes \mW_j^{-1}) \vecop(\mZ_j^\top ) \right\}^2 \mid \mW_j\right]
  %\\
  %&=
  %\frac{n^2}{p^2}
  %\left[
  %(\tr(\mI_{k_j} \otimes \mW_j^{-1}))^2
  %+
  %2\tr\left((\mI_{k_j} \otimes \mW_j^{-1})^2\right)
  %\right]
  %\\
  \\
  &=
  \frac{n^2}{p^2}
  \left\{
  k_j^2(\tr(\mW_j^{-1}))^2
  +
  2k_j\tr(\mW_j^{-2})
  \right\}
\end{align*}
follows directly from Theorem 11.22 of \cite{schott2016matrix}.
Besides, from the result of \cite{watamori1990moments},
\begin{align*}
  E\left[
    \begin{matrix}
      (\tr (\mW_j^{-1}))^2 \\
      \tr (\mW_j^{-2})
    \end{matrix}
  \right]
  &=
  \frac{
    \begin{pmatrix}
      n - k_j - p - 2 & 2 \\
      1 & n - k_j - p - 1
    \end{pmatrix}
    \begin{pmatrix}
      (\tr (\mI_p))^2 \\
      \tr (\mI_p)
    \end{pmatrix}
  }{(n - k_j - p)(n - k_j - p - 1)(n - k_j - p - 3)}
\end{align*}
holds.
Therefore, we have
\begin{align*}
  E[Y_{n, p}^2]
  &=
  \frac{n^2}{p^2} E\left[k_j^2(\tr (\mW_j^{-1}))^2 + 2k_j \tr(\mW_j^{-2})\right]
  \\
  &=
  \frac{n^2 k_j \{(n - k_j - p - 2)k_jp + 2(n - 1)\}}{p(n - k_j - p)(n - k_j - p - 1)(n - k_j - p - 3)}
  ,
\end{align*}
which converges, and so it is bounded.
Now, we can apply Lemma \ref{lem:e_log} and obtain
\begin{align} \label{eq:tilde}
  0
  \leq
  E^{\mY, \tilde\mY}\left[
    \log \left\{
      1
      +
      \frac{(1 - w_j) f_\pi(\tilde \mY \mid \hat\mSigma_j)}{w_j f(\tilde \mY \mid \mX_j \hat\mTheta_j, \hat\mSigma_j)}
    \right\}
  \right]
  \leq
  E^{\mY, \tilde \mY} \left[\log (1 + a_{n, p} \exp(p Y_{n, p}))\right]
  =
  o(p)
  .
\end{align}
In a similar way, we have
\begin{align} \label{eq:not_tilde}
  E^{\mY}\left[
    \log \left\{
      1
      +
      \frac{(1 - w_j) f_\pi(\mY \mid \hat\mSigma_j)}{w_j f(\mY \mid \mX_j \hat\mTheta_j, \hat\mSigma_j)}
    \right\}
  \right]
  =
  o(p)
  .
\end{align}
Combining (\ref{eq:bias}), (\ref{eq:bias_aicc}), (\ref{eq:tilde}), and (\ref{eq:not_tilde}) yields
\begin{align*}
  b_j
  =
  \frac{np(2k_j + p + 1)}{n - p - k_j - 1}
  +
  o_(p)
  .
\end{align*}

\subsection{Proof of Lemma \ref{lem:bf}} \label{subsec:lem:bf}
In a similar way to \ref{subsec:thm:bias}, it is easy to get
\begin{align*}
  0
  \leq
  \frac{(1 - w_j) f_\pi(\mY \mid \hat\mSigma_j)}{w_j f(\mY \mid \hat\mTheta_j, \hat\mSigma_j)}
  \leq
  \frac{(1 - w_j)\ell_{n, p}^{pk_j/2}}{w_j n^{pk_j/2}} \pi_M
  \left(
    \frac
      {2\pi}
      {\left|\frac{\mX_j^\top \mX_j}{n}\right|^{1/k_j} \ell_{n, p}}
  \right)^{pk_j/2}
  ,
\end{align*}
and the right-hand side is $o_p(1)$ from the assumptions given in the statement.
From this and the continuity of $\log,$ the statement holds.

\subsection{Proof of Theorem \ref{thm:cons_ic}} \label{subsec:thm:ic_cons_ls}
In \cite{yanagihara2015consistency}, we can find a proof of the statement when $p$ is unbounded, so we show the consistency property under the condition (\ref{cond:ls}) of Theorem \ref{thm:cons_ic}.

The probability of selecting the true model can be evaluated as
\begin{align*}
  \pr(\hat j_\ast = j_\ast)
  =
  \pr(\forall j \in \cJ\setminus\{j_\ast\},~ IC_m(j) > IC_m(j_\ast))
  \geq
  1 - \sum_{j \in \cJ \setminus \{j_\ast\}} \pr(IC_m(j) \leq IC_m(j_\ast))
  ,
\end{align*}
so it is sufficient to show that
\begin{align} \label{eq:consist_each}
  \limlh \pr(IC_m(j) > IC_m(j_\ast))
  =
  1
\end{align}
holds for all $j \in \cJ \setminus \{j_\ast\}$.

At first, let $j \in \cJ_+ \setminus \{j_\ast\}.$
Notice that $\hat \mSigma_\ast$ admits the following decomposition:
\begin{align*}
  n \hat \mSigma_\ast
  =
  n \hat \mSigma_j
  +
  \mY^\top  (\mP_j - \mP_\ast) \mY
  .
\end{align*}
It is easy to get
\begin{gather*}
  n \hat \mSigma_j
  \sim
  W_p(n - k_j, \mSigma_\ast)
  ,
  \ \
  \mY^\top  (\mP_j - \mP_\ast) \mY
  \sim
  W_p(k_j - k_\ast, \mSigma_\ast)
  ,
\end{gather*}
and these are independent.
By Corollary 1 of \cite{yanagihara2017high},
\begin{align*}
  \frac{n}{p} \log \frac{|\hat \mSigma_j|}{|\hat \mSigma_\ast|}
  &=
  - \frac{n}{p} \log \frac{|n \mSigma_\ast^{-1/2}\hat \mSigma_j\mSigma_\ast^{-1/2} + \mSigma_\ast^{-1/2}\mY^\top  (\mP_j - \mP_\ast) \mY\mSigma_\ast^{-1/2}|}{|n\mSigma_\ast^{-1/2}\hat \mSigma_j\mSigma_\ast^{-1/2}|}
  \\
  &=
  -
  \frac{Z}{p}
  +
  (k_j - k_\ast) \left(1 + \frac{\log(1 - c_{n, p})}{c_{n, p}}\right) + o_p(1)
\end{align*}
where $Z$ is a random variable distributed according to the $\chi^2$ distribution with $(k_j - k_\ast)p$ degrees of freedom.
Since $p$ is bounded, $c_0 = 0$ holds, and we have
\begin{align} \label{eq:odr_ratio_+}
  \frac{n}{p} \log \frac{|\hat \mSigma_j|}{|\hat \mSigma_\ast|}
  =
  O_p(1)
  .
\end{align}
Since we assume (LS-2),
\begin{align*}
  \frac{1}{p(m(j) - m(j_\ast))} (IC_m(j) - IC_m(j_\ast))
  =
  \frac{1}{m(j) - m(j_\ast)} \times \frac{n}{p} \log \frac{|\hat \mSigma_j|}{|\hat \mSigma_\ast|}
  +
  \frac{1}{p}
  =
  \frac{1}{p}
  +
  o_p(1)
  ,
\end{align*}
which implies (\ref{eq:consist_each}) for $j \in \cJ \setminus \{j_\ast\}.$

Next, let $j \in \cJ_-,$ and $j_+$ denotes $j \cup j_\ast$.
By the same procedure as \cite{yanagihara2015consistency}, we have the following decomposition:
\begin{align} \label{eq:decomp}
  \log \frac{|\hat \mSigma_j|}{|\hat \mSigma_{j_+}|}
  =
  -
  \log\frac{|\mU_1|}{|\mU_1 + \mU_2|}
  -
  \log\frac{|\mU_3|}{|\mU_3 + \mU_4|}
\end{align}
where
\begin{gather*}
  \mU_1 \sim W_{\gamma_j}(n - k_j - p, \mI_{\gamma_j})
  ,
  \ \
  \mU_2 \sim W_{\gamma_j}(p, \mI_{\gamma_j}; \mDelta_j^2)
  ,
  \\
  \mU_3 \sim W_{k_{j_+} - k_j - \gamma_j}(n - k_j - \gamma_j - p, \mI_{k_{j_+} - k_j - \gamma_j})
  ,
  \ \
  \mU_4 \sim W_{k_{j_+} - k_j - \gamma_j}(p, \mI_{k_{j_+} - k_j - \gamma_j})
  .
\end{gather*}
For the first term,
\begin{align*}
  - \log\frac{|\mU_1|}{|\mU_1 + \mU_2|}
  &=
  \log |\mDelta_j^{-1} (\mU_1 + \mU_2) \mDelta_j^{-1}|
  -
  \log |\mDelta_j^{-1} \mU_1 \mDelta_j^{-1}|
\end{align*}
holds.
From the definition of the noncentral Wishart matrix, we obtain the following decomposition:
\begin{align*}
  \mU_2
  =
  (\mZ + \mGamma_j)^\top  (\mZ + \mGamma_j)
\end{align*}
where $\mZ \sim N_{p \times \gamma_j}(\mO_{p \times \gamma_j}, \mI_{\gamma_j} \otimes \mI_p)$.
Notice that
\begin{align*}
  E\left[\tr (\mDelta_j^{-1} \mZ^\top \mZ \mDelta_j^{-1})\right]
  =
  p \tr (\mDelta_j^{-2})
  ,
  \ \
  E\left[||\mDelta_j^{-1} \mGamma_j^\top  \mZ \mDelta_j^{-1}||^2\right]
  =
  \gamma_j \tr (\mDelta_j^{-2})
\end{align*}
where $||\cdot||^2$ denotes the Frobenius norm.
From (\ref{ass:ncm}) of Assumption \ref{ass:cons_ass}, it follows that $\tr (\mDelta_j^{-2}) \leq \frac{\gamma_j}{\lambda_j},$ and that it converges to zero.
Hence, we have
\begin{align*}
  \mDelta_j^{-1} \mU_2 \mDelta_j^{-1}
  =
  \mDelta_j^{-1} \mZ^\top \mZ \mDelta_j^{-1} + \mDelta_j^{-1} \mGamma_j^\top  \mZ \mDelta_j^{-1} + \mDelta_j^{-1} \mZ^\top  \mGamma_j \mDelta_j^{-1} + \mDelta_j^{-1} \mGamma_j^\top  \mGamma_j \mDelta_j^{-1}
  \overset{p}{\to}
  \mI_{\gamma_j}
  ,
\end{align*}
Therefore, it follows that
\begin{align*}
  - \log\frac{|\mU_1|}{|\mU_1 + \mU_2|}
  =
  \log |\mI_{\gamma_j} + \mDelta_j^{-1} \mU_1 \mDelta_j^{-1}|
  -
  \log |\mDelta_j^{-1} \mU_1 \mDelta_j^{-1}| + o_p(1)
  .
\end{align*}
By the Minkowski inequality (see, e.g., \cite{schott2016matrix}),
\begin{align*}
  \log |\mI_{\gamma_j} + \mDelta_j^{-1} \mU_1 \mDelta_j^{-1}|
  \geq
  \log \left(1 + |\mDelta_j^{-1} \mU_1 \mDelta_j^{-1}|\right)
\end{align*}
holds.
Here notice the following inequality:
\begin{align*}
  \log(1 + x) - \log x
  =
  \int_0^{\frac{1}{x}} \frac{1}{1 + y} dy
  \geq
  \frac{1}{1 + x}
\end{align*}
for $x > 0$.
These imply
\begin{align*}
  - \log\frac{|\mU_1|}{|\mU_1 + \mU_2|}
  &\geq
  \log \left(1 + |\mDelta_j^{-1} \mU_1 \mDelta_j^{-1}|\right)
  -
  \log |\mDelta_j^{-1} \mU_1 \mDelta_j^{-1}| + o_p(1)
  \\
  &\geq
  \frac{1}{1 + |\mDelta_j^{-1} \mU_1 \mDelta_j^{-1}|} + o_p(1)
  .
\end{align*}
It follows that
\begin{align*}
  |\mDelta_j^{-1} \mU_1 \mDelta_j^{-1}|
  =
  \frac{|\mU_1 / n|}{|\mDelta_j^2 / n|}
  \leq
  \frac{|\mU_1 / n|}{(\lambda_j / n)^{\gamma_j}}
  =
  \left(\frac{\lambda_j}{n}\right)^{-\gamma_j} + o_p(1)
\end{align*}
since
\begin{align*}
  \frac{\mU_1}{n}
  \overset{p}{\to}
  \mI_{\gamma_j}
\end{align*}
holds.
Therefore, we have
\begin{align*}
  - \log\frac{|\mU_1|}{|\mU_1 + \mU_2|}
  \geq
  \frac{1}{1 + (\lambda_j / n)^{-\gamma_j}} + o_p(1)
  .
\end{align*}
From (\ref{ass:ncm}) of Assumption \ref{ass:cons_ass}, the denominator of the first term of the right-hand side is bounded above, so the left-hand side is positive in probability; that is, for some $\varepsilon > 0,$
\begin{align} \label{eq:odr_1_2}
  - \log\frac{|\mU_1|}{|\mU_1 + \mU_2|}
  >
  \varepsilon + o_p(1)
\end{align}
holds.
By the boundedness of $p,$ it follows that
\begin{align*}
  \frac{\mU_3}{n}
  \overset{p}{\to}
  \mI_{k_{j_+} - k_j - \gamma_j}
  ,
  \ \
  \frac{\mU_4}{n}
  \overset{p}{\to}
  \mO_{(k_{j_+} - k_j - \gamma_j) \times (k_{j_+} - k_j - \gamma_j)}
  .
\end{align*}
Hence, we have
\begin{align} \label{eq:odr_3_4}
  - \log\frac{|\mU_3|}{|\mU_3 + \mU_4|}
  =
  - \log\frac{|\mU_3 / n|}{|\mU_3 / n + \mU_4 / n|}
  =
  o_p(1)
  .
\end{align}
Furthermore, from (\ref{eq:odr_ratio_+}), we have
\begin{align} \label{eq:odr_ratio_-}
  \log \frac{|\hat \mSigma_{j_+}|}{|\hat \mSigma_\ast|}
  =
  \frac{p}{n} \times \frac{n}{p} \log \frac{|\hat \mSigma_{j_+}|}{|\hat \mSigma_\ast|}
  =
  o_p(1)
  .
\end{align}
From (\ref{eq:decomp}), (\ref{eq:odr_1_2}), (\ref{eq:odr_3_4}), and (\ref{eq:odr_ratio_-}), it follows that
\begin{align*}
  \log \frac{|\hat \mSigma_{j}|}{|\hat \mSigma_\ast|}
  >
  \varepsilon + o_p(1)
  .
\end{align*}
Therefore, we have,
\begin{align*}
  \frac{1}{n} (IC_m(j) - IC_m(j_\ast))
  >
  \varepsilon + o_p(1)
\end{align*}
from (LS-1), and this implies (\ref{eq:consist_each}) for $j \in \cJ_-.$

\end{document}